\documentclass[12pt]{l4dc2021}

\title{Learning-Enabled Robust Control with Noisy Measurements}
\usepackage{times}
\usepackage{tikz, pgfplots}
\usepackage{todonotes}

\usetikzlibrary{positioning, shapes, arrows, calc}
\tikzstyle{block} = [draw, rectangle, line width=0.7mm]

\newcommand{\bmat}[1]{\begin{bmatrix} #1 \end{bmatrix}}
\newcommand{\R}{\mathbb{R}}
\newcommand{\x}{\mathbf{x}}
\renewcommand{\H}{\mathcal H}

\author{%
	\Name{Olle Kjellqvist} \Email{olle.kjellqvist@control.lth.se}\\
	\Name{Anders Rantzer} \Email{anders.rantzer@control.lth.se}\\
	\addr {Automatic Control LTH \\ Lund University \\ Box 118 \\ SE-221 00 Lund \\ Sweden}
}

\begin{document}

\maketitle

\begin{abstract}%
	We present a constructive approach to bounded $\ell_2$-gain adaptive control with noisy measurements for linear time-invariant scalar systems with uncertain parameters belonging to a finite set. The gain bound refers to the closed-loop system, including the learning procedure. The approach is based on forward dynamic programming to construct a finite-dimensional information state consisting of $\mathcal H_\infty$-observers paired with a recursively computed performance metric. We do not assume prior knowledge of a stabilizing controller.
\end{abstract}


\section{Introduction}
	The great control engineer is lazy; her models are simplified and imperfect, the operating environment may be poorly controlled --- yet her solutions perform well. Robust control provides excellent tools to guarantee performance if the uncertainty is small~\cite{Zhou2008}. If the uncertainty is large, one can perform laborious system identification offline to reduce model uncertainty and synthesize a robust controller. An appealing alternative is to trade the engineering effort for a more sophisticated controller, particularly a learning-based component that improves controller performance as more data is collected. However, for such a controller to be implemented, it had better be robust to any prevalent unmodelled dynamics. Currently, there is considerable research interest in the boundary between machine learning, system identification, and adaptive control. For a review, see for example \cite{Matni2019}. Most of the studies concern stochastic uncertainty and disturbances and assume perfect state measurements. Recently, works connecting to worst-case disturbances have started to appear. For example, non-stochastic control was introduced for known systems with unknown cost functions in \cite{Agarwal2019} and extended to unknown dynamics and output feedback, under the assumption of bounded disturbances and prior knowledge of a stabilizing proportional feedback controller in \cite{Simchowitz2020}. In \cite{Dean2019} the authors leverage novel robustness results to ensure constraint satisfaction while actively exploring the system dynamics. In this contribution, the focus is on worst-case models for disturbances and uncertain parameters as discussed in \cite{Basar94}, \cite{Vinnicombe2004} and more recently in \cite{Rantzer2021}, but differ in that we consider output-feedback. See Figure~\ref{fig:prob} for an illustration of the considered problem. Unlike most recent contributions, the approach taken in this paper:
\begin{enumerate}
	\item does not assume prior knowledge of a stabilizing controller. In particular, we allow for uncertain systems that a linear controller cannot stabilize,
	\item assumes that the measurements are corrupted by additive noise,
	\item provides guarantees on the $\ell_2$-gain from disturbance and noise to state for the entire control duration.
\end{enumerate}
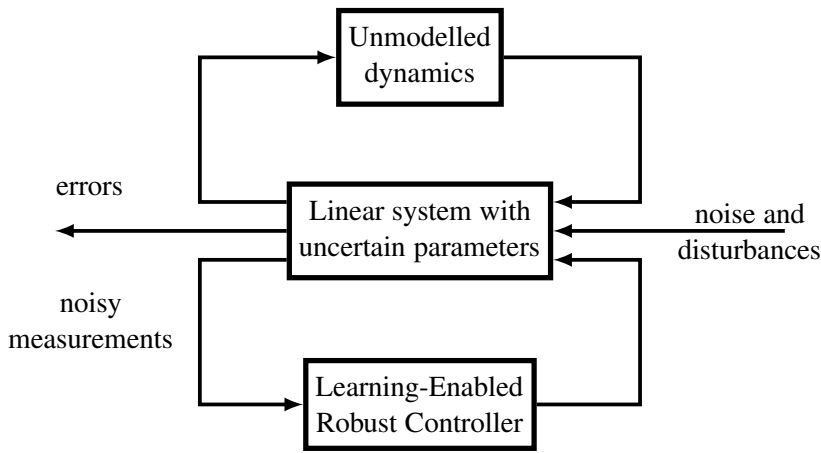
\begin{figure}
	\centering
	\begin{tikzpicture}[node distance = 6em, line width=0.5mm, inner ysep=6pt, >=latex]
	\node [block, align = center] (sys) {Linear system with \\ uncertain parameters};

	\node[block, above of = sys, align = center] (delta) {Unmodelled \\ dynamics};
	\node[block, below of = sys, align = center] (controller) {Learning-Enabled \\ Robust Controller};
	
	\node[left =8em of sys] (out) {};
	\node[right = 8em of sys]  (in) {};

	\node[coordinate, right = 3em of sys, yshift=-1em] (Cout) {};
	\node[coordinate, left =3em of sys, yshift=-1em] (Cin) {};

	\node[coordinate, right = 3em of sys, yshift=1em] (Dout) {};
	\node[coordinate, left =3em of sys, yshift=1em] (Din) {};
	
	\node[below left = 0.5em and 0.5em of Cin, align=center] () {noisy \\ measurements};

	\draw[->] (in) -- node[pos=0.15, align=center] () {noise and \\ disturbances} (sys);
	\draw[->] (sys) -- node[pos=0.85, label={errors}]() {} (out);
	\draw[->] ([yshift=-1em]sys.west) -- (Cin) |- (controller);
	\draw[->] ([yshift=1em]sys.west) -- (Din) |- (delta);
	\draw[<-] ([yshift=-1em]sys.east) -- (Cout) |- (controller);
	\draw[<-] ([yshift=1em]sys.east) -- (Dout) |- (delta);
\end{tikzpicture}
	\caption{For a finite set of linear time-invariant models, the Learning-Enabled Robust Controller minimizes the $\ell_2$-gain from noise and disturbances to errors for any realization of the unknown model parameters. This gain bound guarantees robustness to unmodelled dynamics.}
	\label{fig:prob}
\end{figure}
\subsection{Contributions and Outline}
We formalize the problem of finding a causal output-feedback controller with guaranteed finite $\ell_2$-gain stability that is agnostic to the realization of the system parameters in Section~\ref{sec:prob}. Section~\ref{sec:Full-state} is devoted to characterizing the Learning-Enabled Robust Controller in known or computable quantities. In Theorem~\ref{thm:full_state} we show that ensuring finite $\ell_2$-gain is equivalent to running one $\mathcal H_\infty$-observer for each feasible model, checking the sign of the associated cumulative cost and that each cumulative cost can be computed recursively. We show that it is necessary and sufficient to consider observer-based feedback in Theorem~\ref{thm:observer-based-feedback}. In other words, the history can be compressed to a finite number of recursively computable quantities, growing linearly in the number of feasible models. In Section~\ref{sec:Certainty}, we apply these results to synthesize a controller for an integrator with unknown input sign with a guaranteed bound on the $\ell_2$-gain from noise and disturbances to error.  All results in this paper are in discrete-time and for scalar systems, but sections \ref{sec:prob} and \ref{sec:Full-state} are readily extended to multivariable time-invariant systems.
\section{Notation}
The set of $n\times m$ matrices with real coefficients is denoted $\R^{n \times m}$. The transpose of a matrix $A$ is denoted $A^\top$. For a symmetric matrix $A \in \R^{n\times n}$ and a vector $x \in \R^n$ we use the expression $|x|^2_A$ as shorthand for $x^\top A x$. We write $A \prec (\preceq)\ 0$ to say that $A$ is positive (semi)definite. We refer to the value of a signal $w$ at time $t$ as $w(t)$. The space of square-summable sequences from $\{T_0,T_0 + 1, \ldots, T_f\}$ taking values in $\R$ is denoted  $\ell_2[T_0, T_f]$. For a set $\mathcal S$, we let $\#(\mathcal S)$ be the cardinality. 
\section{Learning-Enabled Control with Guaranteed Finite $\ell_2$ Gain}
\label{sec:prob}
	Given a positive quantity $\gamma > 0$ and a finite set of feasible models $\mathcal M \subset \R^3$, we concern ourselves with the uncertain linear system
\begin{equation}
\begin{aligned}
	x(t+1) & = ax(t) + bu(t) + w(t), \quad x(0) = x_0\\
	y(t) & = cx(t) + v(t),\quad t \geq 0
\end{aligned}
        \label{eq:sys}
\end{equation}
	where the control signal $u(t) \in \R $ is generated by a causal output-feedback control policy
\begin{equation}
	u(t) = \mu_t\left(y(0), y(1), \ldots, y(t)\right).
	\label{eq:mu}
\end{equation}

In \eqref{eq:sys}, $x(t) \in \R$ is the state, $y(t) \in \R$ is the measurement, the model $M:= (a, b, c)$ is unknown but belongs to $\mathcal M$. The noise $v$ and disturbances $w$ satisfy $w,v \in \ell_2([0, T])$ for all $T \geq 0$. We are interested in control that makes the closed-loop system finite gain, with gain from $(w,v)$ to $x$ bounded above by $\gamma$. That is,

\begin{equation}
	\alpha(T) :=\sum_{\tau \leq T+1} x(\tau)^2 - \gamma^2\sum_{\tau \leq T}w(\tau)^2 - \gamma^2\sum_{\tau \leq T+1}v(\tau)^2 - P_Mx(0)^2 \leq 0
        \label{eq:finite_gain_cond}
\end{equation}
must hold for all $T \geq 0$, any admissible disturbances, initial state and the possible realizations $M$ of \eqref{eq:sys}. $P_M$ quantifies prior information on the initial state and is taken as a positive solution to the Riccati equation
\begin{equation}
	P_M = \left(a^2\left(P_M + \gamma^2c^2 - 1\right)^{-1} + \gamma^{-2} \right)^{-1}.
	\label{eq:riccati}
\end{equation}

In this article, we explicitly construct controllers satisfying the finite-gain property and give conditions under which such controllers exist for the case when $c = 1$ and $b = \pm 1$.
\begin{remark}
	The cases $b = -1$ and $b = 1$ cannot be simultaneously stabilized by a static feedback controller when $a \geq 1$
\end{remark}

\begin{remark}
	$P_M$ could be any positive quantity. Our choice leads to stationary observer dynamics, simplifying the coming sections.
\end{remark}

\section{An information-state condition}
\label{sec:Full-state}
In this section we will apply a slight modification to the $\H_\infty$-observer from~\cite{Basar95} to bound \eqref{eq:finite_gain_cond} in a way which leads itself to recursive computation. We need the following lemma:
\begin{lemma}[Past cost]
        \label{lemma:past_cost}
	Given a known model $M = (a,b,c)$, a positive quantity $\gamma$, assume that the Riccati equation \eqref{eq:riccati} has a positive solution $P_M$. For fixed $u \in \ell_2([0, t]),\ y \in \ell_2[0,t])$ and $x(t+1)\in \R$, we have that
	\begin{multline}  
		\sup_{w,v\in \ell_2[0,t], x_0 \in \R}\left\{\sum_{\tau \leq t} x(\tau)^2 - \gamma^2\sum_{\tau \leq t}\left(w(t)^2 + v(t)^2\right) - Px(0)^2:\ \text{subject to \eqref{eq:sys}}\right\}\\
		= -P_M(x(t+1) - \hat x_M(t+1))^2 + l_M(t+1).
		\label{eq:history}
	\end{multline}
	The state observer $\hat x_M(t)$, and the past cost $l_M(t)$ are defined by the recursion
        \begin{align}
		K_M & = \frac{\gamma^2c_M^2}{P_M + \gamma^2c^2 - 1}, \quad \hat w_M(t) = \frac{\hat x_M(t)}{P_M + \gamma^2c^2 - 1},\nonumber \\
		\hat x_M(t+1) & = a\hat x(t) + bu(t) + K_M\left(y(t) - c\hat x(t) \right) + \hat w_M(t), \quad \hat x_M(0) = 0, \label{eq:sep}\\
		l_M(t+1) & = l_M(t) - P_M\hat x_M(t)^2 - \gamma^2(y_t)^2 + \frac{\left(P_M\hat x_M(t) + \gamma^2 cy(t)\right)^2}{P_M + \gamma^2c^2 - 1}, \quad  l_M(0) = 0.\nonumber
        \end{align}
\end{lemma}
\begin{remark}
	The observer form \eqref{eq:sep} makes sense for linear systems where we can design a state-feedback controller and observer separately and then join them together using the separation principle in~\cite{Basar95}. The assumptions for the separation principle are not satisfied in our case, so we find it simpler to use the equivalent form 
	\[
		\hat x_M(t+1) = \hat a_M x(t) + bu(t) + \hat g_M y(t),
	\]
	where $\hat a_M = aP_M / (P_M + \gamma^2c^2 - 1)$ and $\hat g_M = \gamma^2 a c / (P_M + \gamma^2c^2 - 1)$.
\end{remark}
\begin{proof}[Lemma~\ref{lemma:past_cost}]
	The system is equivalent to (6.1) and (6.2) in \cite[p.~243]{Basar95} but with $D_k = \bmat{I & 0}$ and $E_k = \bmat{0 & I}$. Note that the term $-P_Mx(0)^2$ in \eqref{eq:history} ensures that $P_{k+1} = P_k = \ldots = P_M$, i.e. stationarity.  Explicitly computing $l_M(t)$ requires some extra bookkeeping; in $6.35$ the \emph{terms independent of $\xi$ and $w$} is equivalent to $\gamma^2|y_t|^2_{(HH^\top)^{-1}} + |\hat x_t|^2_{P_t} - |u_t|^2_R - l_t$, the notational differences are $(HH^\top) \to N$, $P_t \to K_t$ and $l_t \to c_t$. After application of Lemma 6.2 on $p.~259$ we identify
        \[
                m_k = -|P_t\hat x_t + \gamma^2C^\top(HH^\top)^{-1}y_t|^2_{(P_t + \gamma^2C^\top(HH^\top)^{-1}C - Q)^{-1}} + \gamma^2|y_t|^2_{(HH^\top)^-1} + |\hat x_t|^2_{P_t} - |u_t|^2_R - l_t
        \]
	and conclude $l_M(t+1) = -m_k$.
\end{proof}
	Lemma~\ref{lemma:past_cost} lets us express the worst-case accumulated cost compatible with the dynamics as a function of the past trajectory $(u,y)$ and the next state $x(t+1)$, if the dynamics $M$ of the system \eqref{eq:sys} are known. As $x(t+1)$ changes, so does the set of trajectories $w,v$ that are compatible with $x(t+1)$. In particular, the entire sequence of a maximizing trajectory will change as $x(t+1)$ is varied. With that in mind, it is remarkable that the effect to the accumulated cost is captured completely by the term $-P\left(x(t+1) - \hat x(t+1)^2\right)$. The second term $l(t+1)$ contains the terms of the cost that depend only on past inputs and outputs and is independent of $x(t+1)$.

We will study the value of the left-hand side of \eqref{eq:finite_gain_cond} for each model separately. Define for $M = (a, b, c) \in \mathcal M$, $y\in \ell_2[0,t]$ and an arbitrary output-feedback control policy $\mu$ the quantities
\begin{equation}
	\alpha_M(t) := \sup_{w,v \in \ell_2[0,t], x_0 \in \R} \left \{ \alpha(t): (a,b,c) = M, \text{subject to \eqref{eq:sys} and \eqref{eq:mu}}\right \}
	\label{eq:alpha}
\end{equation}

Then $\max_M\alpha_M(t)$ is the largest possible value of \eqref{eq:finite_gain_cond} at time $t$. In the following theorem, we use Lemma~\ref{lemma:past_cost} to express $\alpha_M$ recursively and construct equivalent conditions using computable quantities.
\begin{theorem}[Information-state condition]
        \label{thm:full_state}
	Given a causal output-feedback control policy$\mu$, a positive quantity $\gamma$, and an uncertainty set $\mathcal M$. Assume that for all $(a, b, c) = M \in \mathcal M$ the Riccati equation
	\begin{equation}
                P_M = \left( \frac{a^2}{P_M + \gamma^2c^2 - 1} + \gamma^{-2} \right)^{-1}
		\label{eq:Riccati}
	\end{equation}
	a positive solution $P_M$ and let
        \[
                \hat a_M = \frac{aP_M}{P_M + \gamma^2c^2 - 1},\qquad \hat g_M = \gamma^2\frac{ac}{P_M + \gamma^2c^2 -1}.
        \]
	Further let
	\begin{align}
                \hat x_M(t+1) & = \hat a_M\hat x_M(t) + bu(t) + \hat g_My(t),\ \hat x_M(0) = 0,
		\label{eq:xhati} \\
		l_M(t+1) & = l_M(t) - P_M\hat x_M(t)^2 - \gamma^2y(t)^2 + \frac{(P_M\hat x_M(t) + \gamma^2c y(t))^2}{P_M + \gamma^2c^2 - 1},\quad l_M(0) = 0.
                \label{eq:li}
        \end{align}
	Then the closed-loop system \eqref{eq:sys}, \eqref{eq:mu} with control $\mu$ is finite gain for any realization $M\in \mathcal M$ if and only if $l_M(t+1) \leq 0$ holds for all $M \in \mathcal{M}$,  $t\geq 0$ and $y \in \ell_2([0, t])$. If $P_M < 1$ for some $M$, $\gamma$ is not an upper bound of the $\ell_2$-gain from disturbance to error.
\end{theorem}

\begin{proof}
	Let $\alpha_M(t)$ be defined as in \eqref{eq:alpha}. Then \eqref{eq:finite_gain_cond} holds for all $(w,v, x_0)$, $M \in \mathcal{M}$ and $T$ if and only if $\alpha_M(T) \leq 0$ for all $M\in \mathcal M$ and $y\in \ell_2[0, T]$. We now apply Lemma~\ref{lemma:past_cost} to express $\alpha_M(t)$ in the known quantities $\hat x_M(t)$, $P_M$ and $l_M(t)$\footnote{We let subscript $M$ denote quantities using $(a, b, c) = M$.}:
\[
\begin{aligned}
	\alpha_M(t) & = \sup_{x(t),v(t)\in \R}\sup_{w, v\in \ell_2[0,t], x_0\in \R} \Bigg\{x(t)^2 - \gamma^2v(t)^2 + \sum_{\tau \leq t-1}x(\tau)^2 - \gamma^2 \sum_{\tau \leq t-1}\left( w(t)^2 + v(t)^2 \right) \\
		& \quad : x(t+1) = ax(t) + bu(t) +w(t),\ y(t) = cx(t) + v(t),\  (a, b, c) = M\Bigg\}\\
	& = \sup_{x\in \R, v\in \R} \left\{x^2 -\gamma^2v^2 - P_M\left(x - \hat x_M(t)\right)^2  + l_M(t)\right\} \\
                                & = \left(P_M\hat x_M(t) + \gamma^2cy(t)\right)^2/(P_M + \gamma^2c^2 - 1) - P_M\hat x_M^2(t) - \gamma^2y(t)^2 + l_M(t) \\
                                & = l_M(t+1).
\end{aligned}
\]

	Finally, note that if for some $M$, $P_M < 1$, then $l_M(t+1)$ is strictly convex in $y(t)$ and thus unbounded from above.
\end{proof}

From Theorem~\ref{thm:full_state} we see that the observer states $\hat x_M(t)$ and cumulative objectives $l_M(t+1)$ contain the information necessary and sufficient to evaluate the finite-gain condition \eqref{eq:finite_gain_cond}. In other words, we can tell everything we need about the current state of affairs by running one $\H_\infty$ observer and computing $l_M$(t+1) for each model $M$ in parallel; \emph{but is it sufficient to consider observer-based feedback for control? If so, is it also necessary?.}  the next theorem, we show that the observer states and cumulative objectives contain precisely the information required to synthesize a finite-gain control policy.

\begin{theorem}[Observer-based feedback]
	\label{thm:observer-based-feedback}
	Given a positive quantity $\gamma > 0 $ and an uncertainty set $\mathcal M \in \R^3$. The following are logically equivalent.
	\begin{enumerate}
		\item[(i)] There exists a causal output-feedback control policy $\mu^\star$ such that the closed-loop system \eqref{eq:sys} and \eqref{eq:mu} is finite-gain.
		\item[(ii)] There exist observers $(\hat x_M, l_M)$ for each model $m \in \mathcal M$ generated by \eqref{eq:xhati}, \eqref{eq:li} and an observer-based control policy $\eta^\star$ 
			\[
				u(t) = \eta^\star\left\{\left(\hat x_M(t), l_M(t+1), y(t)\right): m\in \mathcal M \right\},
			\]
			such that $l_M(t+1) \leq 0$ for all $m \in \mathcal M$, $y \in \ell_2[0, t]$ and $t\geq 0$.
	\end{enumerate}
	
	If $\eta^\star$ satisfies (ii), the following control policy satisfies (i):
	\begin{equation}
		\mu^\star_t\left( y(0), y(1) \ldots, y(t) \right) =\eta^\star\left\{\left(\hat x_M(t), l_M(t+1), y(t)\right): m\in \mathcal M \right\} 
		\label{eq:mueta}
	\end{equation}
\end{theorem}

\begin{remark}
	By compressing the past trajectory to a finite set of cumulative performance quantities $l_M$, policies of this type learns the actual dynamics of the system as time goes on. This leads to a kind of multi-observer controller. The architecture is illustrated in~\ref{fig:observer-based-feedback}.
\end{remark}

\begin{figure}
	\centering
	\begin{tikzpicture}[line width=0.5mm, inner ysep=6pt, >=latex]
	
	\node [draw, rectangle, minimum height=9em, line width=0.7mm, align = center] (C) {Observer-based \\ Controller: \\ $\eta$};
	
	\node[block, left = 9em of C, minimum height=3em] (obs2) {Observer 2};
	\node[block, minimum height=3em, above= 1.5em of obs2] (obs1) {Observer 1};
	\node[below = -.5em of obs2] (dots) {$\vdots$};
	\node[block, below= 0em of dots, minimum height = 3em] (obsK) {Observer $K$};

	\node[coordinate, right = 6em of obs1] (bend1) {};
	\node[coordinate, right = 6em of obsK] (bendK) {};

	\node[left = 9em of obs2] (in) {$y$};
	\node[right = 6em of C, color=black!10!blue] (out) {$u$};
	\node[coordinate, right = 6em of in] (split) {};
	\node[coordinate, above = 2em of obs1] (ff) {};

	\draw[->, label distance=-0.5em] ([yshift=1em]obs1.east) -- node[label=below:{$(\hat x_1, l_1)$}] {} ([yshift=1em]bend1.center) |- ([yshift=3.5em]C.west);
	\draw[<-, black!10!blue, label distance=-0.5em] ([yshift=-1em]obs1.east) -- node[xshift=0.5em, color = black!10!blue, label=below:{$u$}] {} ([yshift=-1em, xshift=-1em]bend1.center) |- ([yshift=2.5em]C.west);
	\draw[->, label distance=-0.5em] ([yshift=1em]obs2.east) -- node[pos = 0.35, label=below:{$(\hat x_2, l_2)$}] {} ([yshift=1em]C.west);
	\draw[<-, label distance=-0.5em, black!10!blue] ([yshift=-1em]obs2.east) -- node[pos = 0.35,color = black!10!blue, label=below:{$u$}] {}  ([yshift=-1em]C.west);

	\draw[->, label distance=-0.5em] ([yshift=1em]obsK.east) -- node[xshift=0.5em, label=below:{$(\hat x_K, l_K)$}] {} ([yshift=1em, xshift=-1em]bendK.center) |- ([yshift=-2.5em]C.west);
	\draw[<-, black!10!blue, label distance=-0.5em] ([yshift=-1em]obsK.east) -- node[color=black!10!blue, label=below:{$u$}] {} ([yshift=-1em]bendK.center) |- ([yshift=-3.5em]C.west);

	\draw[->] (in) -- (obs2);
	\draw[->] (split) |- (obs1);
	\draw[->] (split) |- (obsK);
	\draw[->] (split) |- (ff) -| (C);
	\draw[->, black!10!blue] (C) -- (out);

	\coordinate(b1) at ($(in.east |- ff.north) + (2em, 2em)$);
	\coordinate(b2) at ($(out.west |- obsK.south) + (-2em, -2em)$);
	\draw[rounded corners, draw=black!10!blue, dashed] (b1) rectangle (b2);
	\node[below =1em of C, align=center, color=black!10!blue] () {Causal output-feedback \\ controller: $\mu$};

\end{tikzpicture}
	\caption{Illustration of the controller architecture in Theorem~\ref{thm:observer-based-feedback} for uncertainty sets consisting of $K$ linear models. The controller $\eta$ only considers the current state of the observers.}
	\label{fig:observer-based-feedback}
\end{figure}
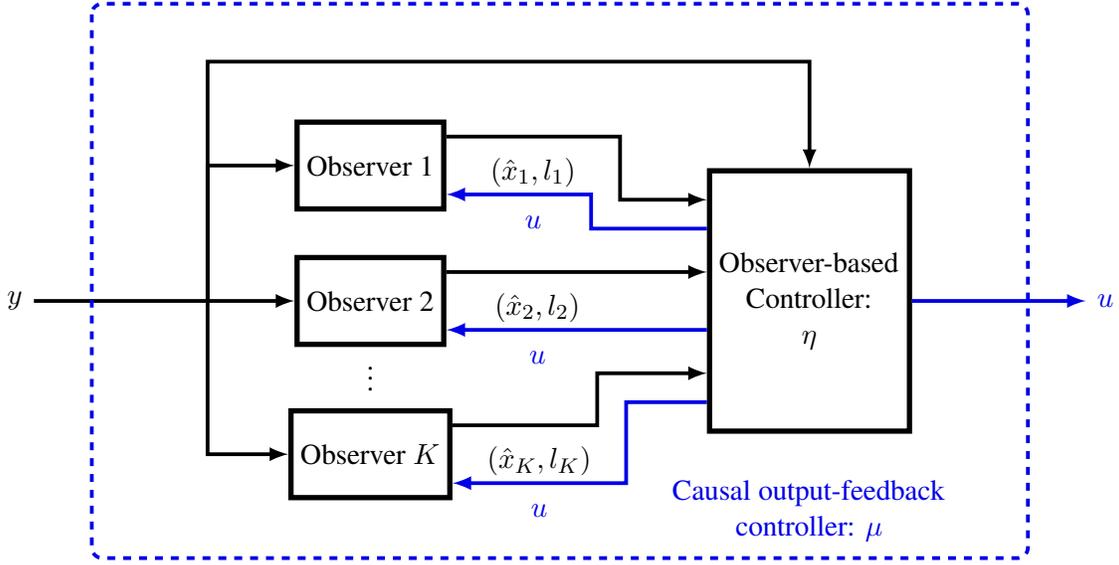
\begin{proof}{Theorem~\ref{thm:observer-based-feedback}}
	\emph{(ii) implies (i)} follows from that $\hat x_M(t), l_M(t+1)$ depend causally on $y$, thus the observer-based control policy is a special case of causal feedback control policies. By assumption, $l_M(T) \leq 0$ for all $T$, $M$ and $y\in \ell_2[0, T]$ for the controller \eqref{eq:mueta}, which we know implies that the system is finite gain by Theorem~\ref{thm:full_state}.

	\emph{(i) implies (ii)}: Assume that the controller $\mu^\star$ fulfills (i). By the construction of \eqref{eq:finite_gain_cond} the Riccati equations have positive solutions $P_M$, therefore the assumptions of Theorem~\ref{thm:full_state} are fulfilled and there exist observers $\hat x_M$ and $l_M$ generated by \eqref{eq:xhati} and \eqref{eq:li}. Define the set of feasible generating trajectories given observer states $\hat x_M(t)$, $l(t)$ and current measurement $y(t)$:
	\begin{multline*}
		\mathcal T\left \{(\hat x_M(t), \hat l_M(t+1), y(t)) : M \in \mathcal M \right \} := \Big \{ (\breve y(\tau))_{\tau = 0}^T : \breve x_M(T) = \hat x_M(t), \breve y(T) = y(t) \\ \breve l_M(T+1) = l_M(t+1), (\breve x_M, \breve l_M) \text{ generated by } \breve y \text{ and } u(\tau) = \mu^\star(\breve y(0), \ldots \breve y(\tau)) \Big \}.
	\end{multline*}
	Then $\mathcal T\left \{(\hat x(0), l_M(1), y(0)): M \in \mathcal M \right \}$ is nonempty since it is compatible with any trajectory of length $1$ such that $\breve y(0) = y(0)$. Fix $t \geq 0$ and observer states $\hat x_M(t), l_M(t+1)$ and measurement $y(t)$. Assume that $\mathcal T\left\{\hat x_M(t), l_M(t), y(t)): M \in \mathcal M\right \}$ is non empty. Then there exists a sequence $\breve y$, and final time $T$ so that $l_M(t+1) = \alpha_M(T)$ with $\alpha_M(t)$ as in \eqref{eq:alpha} generated by $\breve y$ and the controller $u(\tau) = \mu^\star(\breve y(0), \ldots, \breve y(\tau))$. By assumption, $l_M(t+1) = \alpha_M(T) \leq 0$. Taking 
	\begin{align*}
		\eta^\star\left\{(\hat x_M(t), l_M(t+1), y(t)):M\in \mathcal M \right\} = \mu^\star(\breve y), \\
	\end{align*}
		for some $\breve y, \in \mathcal T\left\{(\hat x_M(t) l_M(t+1), y(t)): M \in \mathcal M\right\}$ ensures that $\mathcal T$ will be nonempty the next time step. 
		By induction $\mathcal T$ will be nonempty for all $T\geq 0$ and thus $u$ is well defined and $l_M(T) \leq 0$ for all $T$.
\end{proof}

\section{Certainty equivalence control}
\label{sec:Certainty}
We will now leverage these results to synthesize a control policy for the case when the pole $a\in \R$ is known, $b = \pm 1$ and $c = 1$. Emboldened by Theorem~\ref{thm:observer-based-feedback} we will construct a simple observer-based supervisory controller in the following way: We will run two observers in parallel corresponding to the cases $b = \pm 1$. The supervisor will monitor the cumulative objectives $l_{-1}(t)$ and $l_1(t)$ and determine which observer and model to use for computing the control signal. The policy computes the control signal as if the selected model were true.  Let $i \in \{-1, 1\}$ index the observers. The Riccati equations~\eqref{eq:Riccati} reduce to
\begin{equation}
                P_i = P = \frac{1}{2} (1 - \gamma^2a^2) + \sqrt{\gamma^2(-1 + \gamma^2) + (\gamma^2a^2 -1)^2 / 4}.
		\label{eq:P}
\end{equation}
	Construct the observers $\hat x_i$ and cumulative objectives $l_i$ using \eqref{eq:xhati} and \eqref{eq:li} with $b_i = i$ and
\[
	\hat a_i = \hat a = \frac{aP}{P + \gamma^2 - 1}, \quad \hat g_i = \hat g = \frac{\gamma^2a}{P + \gamma^2 - 1}.
\]
Define the \emph{certainty-equivalence dead-beat controller} as the function
\begin{equation}
	u(t) = \begin{cases}
                -(\hat a\hat x_1(t) + \hat g y(t))  & \text{if}\ l_1(t+1) \geq l_{-1}(t+1) \\
		\hat a\hat x_{-1}(t)  + \hat g y(t)& \text{if}\ l_1(t+1) < l_{-1}(t+1).
        \end{cases}
        \label{eq:certainty_equivalence}
\end{equation}
The dead-beat controller\footnote{The controller is dead-beat for the observer state corresponding to the model with the hightest cumulative cost. The observers themselves are not dead-beat.} ensures that for every $t$, either $\hat x_1(t)$ or $\hat x_{-1}(t)$ will be zero. This simplifies the observer dynamics $\hat x$ and the cost associated with the history $l$. We summarize the properties in the following proposition.
\begin{proposition}
        \label{prop:prop}
	With $\hat a$, $\hat g$, $P$ as above, $\hat x_i$ and $l_i$ as in \eqref{eq:xhati} and \eqref{eq:li}, and the control signal given by \eqref{eq:certainty_equivalence}, let
        \[
                \hat x(t+1) = \hat a \hat x(t) + 2 \hat gy(t), \quad \hat x(0) = 0.
        \]
        Then the following is true:
	\begin{align}
			1: \quad  &
				\hat x_1(t) = \begin{cases}
				0, & \text{if } l_1(t) \geq l_{-1}(t) \\
				\hat x(t) , & \text{if } l_1(t) < l_{-1}(t)
	\end{cases}, 
	\quad
                \hat x_{-1}(t) = \begin{cases}
                        \hat x(t) , & \text{if } l_1(t) \geq l_{-1}(t) \\
                        0, & \text{if } l_1(t) < l_{-1}(t),
                \end{cases} \nonumber  \\
                                2: \quad &
                \begin{cases}
			l_1(t+1) & = \begin{cases}
                        l_1(t) - \gamma^2y(t)^2 + \frac{(\gamma^2 y(t))^2}{P + \gamma^2 - 1} & \text{if } l_1(t) \geq l_{-1}(t) \\
                        l_1(t) - P\hat x(t)^2 - \gamma^2y(t)^2 + \frac{(P\hat x(t) + \gamma^2 y(t))^2}{P + \gamma^2 - 1}, & \text{if } l_1(t) < l_{-1}(t)
                \end{cases} \\
			l_{-1}(t+1) &= \begin{cases}
                                l_{-1}(t) - P\hat x(t)^2 - \gamma^2y(t)^2 + \frac{(P\hat x(t) + \gamma^2 y(t))^2}{P + \gamma^2 - 1}, & \text{if } l_1(t) \geq l_{-1}(t) \\
                                l_{-1}(t) - \gamma^2y(t)^2 + \frac{(\gamma^2 y(t))^2}{P + \gamma^2 - 1}, & \text{if } l_1(t) < l_{-1}(t)
                \end{cases}
		\end{cases} \label{eq:l+}
\end{align}
\end{proposition}
\begin{proof}
	We start by proving the first claim. Consider the case when $l_1(t+1) \geq l_{-1}(t+1)$. Then $\hat x_1(t+1) = 0$ and $\hat x_{-1}(t+1) = \hat a(\hat x_1(t) + \hat x_{-1}(t)) + 2\hat gy(t)$. The case when $l_1(t+1) < l_{-1}(t+1)$ is similar. Taking $\hat x(t) = \hat x_1(t) + \hat x_{-1}(t)$ completes the proof. To see that the second claim is true, note that if $l_1(t) \geq l_{-1}(t)$ then $\hat x_1(t) = 0$ and $\hat x_{-1}(t) = \hat x(t)$. The claim follows by substitution into \eqref{eq:li}.
\end{proof}

\subsection{Conditions for finite-gain stability}
This section determines sufficient conditions for the certainty-equivalence controller to guarantee a gain-bound of at most $\gamma$. We first give conditions on $l_1(t)$ and $l_{-1}(t)$ such that both quantities are negative for the next time step. We will then give conditions on $\gamma$ so that the negativity conditions hold for all $t$. We summarize the non-negativity conditions in the following Lemma.
\begin{lemma}
        \label{lemma:rec_cond}
	Given $P > 1$, $\gamma > 0$, $\hat x(t) \in \R$, $l_1(t)$ and $l_{-1}(t)$. Assume that $\max_{i\in\{-1, 1\}}l_i(t) \leq 0$ and that
        \[
                \min_il_i(t) \leq -\frac{P}{P-1}\hat x(t)^2.
        \]
	Then with $l_i(t+1)$ as in \eqref{eq:l+}, it holds that $l_i(t+1) \leq 0$ for $i\in \{1, -1\}$.
\end{lemma}

\begin{proof}[Lemma~\ref{lemma:rec_cond}, full]
	We will give the proof for the case $0 \geq l_1(t) \geq l_{-1}(t)$. The case $0 \geq l_{-1}(t) \geq l_1(t)$ is similar. Note that $l_1(t+1)$ and $l_{-1}(t+1)$ are concave in $y(t)$ if and only if
        \[
        \frac{1}{\gamma^2} \geq \frac{1}{P + \gamma^2 - 1} \iff  P + \gamma^2 - 1 \geq  \gamma^2,
        \]
        and we conclude that $l_1(t+1)$ and $l_{-1}(t+1)$ are bounded from above if and only if $P \geq 1$. Secondly, we see that $ l_1(t+1) = l_1(t) -cy^2 \leq 0$ for some positive constant $c$. Finally, let $X = P + \gamma^2 - 1$ and consider
        \[
                \begin{aligned}
                        \max_{y(t)} l_{-1}(t+1) &=\max_{y(t)}\left\{ l_{-1}(t) - P\hat x(t)^2 - \gamma^{-2}\left(\gamma^2 y(t)\right)^2 + (P\hat x(t) + \gamma^2 y(t))^2 / X \right \} \\
                                & = \max_{y(t)}\left\{ l_{-1}(t) + \left(-\gamma^{-2} + X^{-1} \right)\left(\gamma^2 y(t) \right)^2 + 2X^{-1}P\hat x(t) \gamma^2y(t) - (P - P^2/X)\hat x(t) \right \} \\
                                & = l_{-1}(t) - \left( \frac{X^{-2}P^2}{-\gamma^{-2} + X^{-1}} + P - P^2/X \right)\hat x(t)^2 \\
                                & = l_{-1}(t) - \frac{\gamma^2P^2/X + P(\gamma^2 - X) - P^2/X (\gamma^2 - X)}{\gamma^2 - X}\hat x(t)^2 \\
                                & = l_{-1}(t) - \frac{P(\gamma^2 - X) + P^2}{\gamma^2 - X}\hat x(t)^2 \\
                                & = l_{-1}(t) - \frac{ P(1 - P) + P^2}{1 -P}\hat x(t)^2 \\
                                & = l_{-1}(t) + \frac {P}{P-1}\hat x(t)^2
                \end{aligned}
        \]
        Which is negative if and only if $l_{-1}(t) \leq -\frac{P}{P-1}\hat x(t)^2$.
\end{proof}

Next we give conditions on $\gamma$ so that the assumptions in Lemma~\ref{lemma:rec_cond} are fulfilled for all $t$. This is illustrated in Figure~\ref{fig:quadfuns}, where subfigure (a) illustrates a case where $l_1(t+1)$ and $l_{-1}(t+1)$ cannot simultaneously be greater than $-\frac{P}{P-1}\hat x(t+1)^2$ and subfigure (b) illustrates the case when the condition is not guaranteed to hold for the next time step. For values of $\gamma$ so that the system behaves as in Figure~\ref{fig:quadfuns} (a), if the assumptions are fulfilled for some $t$, then (by induction) they will be fulfilled for all $T \geq t$. This is formalized in the next theorem.
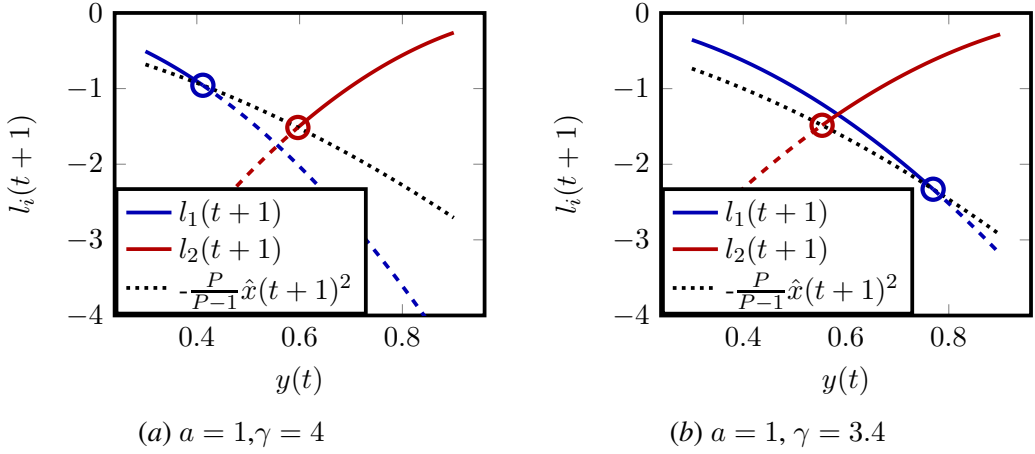
\begin{figure}
	\subfigure[{$a = 1$,$\gamma = 4$}]{%
		\def \g {4}
		\def \a {1}
\def \xh {1}
\def \ymi {0.3}
\def \yma {0.9}
\begin{tikzpicture}[%
	declare function={
		g = \g;
		a = \a;
		xh = \xh;
		P = -(g^2 * a^2 - 1)/2 + sqrt((g^2 * a^2 - 1)^2 / 4 + g^2 * (g^2 - 1));
		X = P + g^2 -1;
		Ghat = a * g^2 / X;
		Ahat = a * P / X;
		yf(\x) = -P / 2 / g^2 * xh + \x;
		i1 = yf(P * xh / 2  / (g^2 - 2 * g^2 / (P-1) * sqrt(g^2 - P)));
		i2 = yf(1 / 2 / g^2 * (P + 2 * g^2 - 1) * ((P - 1) - 2 * sqrt(g^2 - P)) / ((P + 1)^2 - 4 * g^2) * P * xh);
		f1y(\x) = (g^2 * \x)^2 / X - g^2 * \x^2;
		f2y(\x) = -P / (P-1) * xh^2 + (P*xh + g^2 * \x)^2 / X - P * xh^2 - g^2 * \x^2;
		f3y(\x) =-P / (P - 1) * (Ahat * xh + 2 * Ghat * \x)^2;
		f1 = f1y(i1);
		f2 = f2y(i2);
	},
	]
	\begin{axis}[%
			line width=0.5mm,
			width = 6.5cm,
			xlabel = {$y(t)$},
			ylabel = {$l_i(t+1)$},
			ymin = -4,
			ymax = 0,
			legend style={at={(0,0)},anchor=south west,legend cell align=left}
			]
		\addplot[black!30!blue, domain=\ymi:i1] {f1y(\x)};
		\addlegendentry{$l_1(t+1)$};
		\addplot[black!30!red, domain=i2:\yma] {f2y(\x)};
		\addlegendentry{$l_2(t+1)$};
		\addplot[domain=\ymi:\yma, dotted] {f3y(\x)};
		\addlegendentry{-$\frac{P}{P-1}\hat x(t+1)^2$};
		\addplot[black!30!blue, domain=i1:\yma, dashed] {f1y(\x)};
		\addplot[black!30!blue, mark=o, ultra thick, mark size = 4pt] coordinates{(i1,f1)};
		\addplot[black!30!red, domain=\ymi:i2, dashed] {f2y(\x)};
		\addplot[black!30!red, mark=o, ultra thick, mark size = 4pt] coordinates{(i2,f2)};
	\end{axis}
\end{tikzpicture}
		}
	\subfigure[$a = 1$, $\gamma = 3.4$]{%
		\def \g {3.4}
		\def \a {1}
\def \xh {1}
\def \ymi {0.3}
\def \yma {0.9}
\begin{tikzpicture}[%
	declare function={
		g = \g;
		a = \a;
		xh = \xh;
		P = -(g^2 * a^2 - 1)/2 + sqrt((g^2 * a^2 - 1)^2 / 4 + g^2 * (g^2 - 1));
		X = P + g^2 -1;
		Ghat = a * g^2 / X;
		Ahat = a * P / X;
		yf(\x) = -P / 2 / g^2 * xh + \x;
		i1 = yf(P * xh / 2  / (g^2 - 2 * g^2 / (P-1) * sqrt(g^2 - P)));
		i2 = yf(1 / 2 / g^2 * (P + 2 * g^2 - 1) * ((P - 1) - 2 * sqrt(g^2 - P)) / ((P + 1)^2 - 4 * g^2) * P * xh);
		f1y(\x) = (g^2 * \x)^2 / X - g^2 * \x^2;
		f2y(\x) = -P / (P-1) * xh^2 + (P*xh + g^2 * \x)^2 / X - P * xh^2 - g^2 * \x^2;
		f3y(\x) =-P / (P - 1) * (Ahat * xh + 2 * Ghat * \x)^2;
		f1 = f1y(i1);
		f2 = f2y(i2);
	},
	]
	\begin{axis}[%
			line width=0.5mm,
			width = 6.5cm,
			xlabel = {$y(t)$},
			ylabel = {$l_i(t+1)$},
			ymin = -4,
			ymax = 0,
			legend style={at={(0,0)},anchor=south west,legend cell align=left}
			]
		\addplot[black!30!blue, domain=\ymi:i1] {f1y(\x)};
		\addlegendentry{$l_1(t+1)$};
		\addplot[black!30!red, domain=i2:\yma] {f2y(\x)};
		\addlegendentry{$l_2(t+1)$};
		\addplot[domain=\ymi:\yma, dotted] {f3y(\x)};
		\addlegendentry{-$\frac{P}{P-1}\hat x(t+1)^2$};
		\addplot[black!30!blue, domain=i1:\yma, dashed] {f1y(\x)};
		\addplot[black!30!blue, mark=o, ultra thick, mark size = 4pt] coordinates{(i1,f1)};
		\addplot[black!30!red, domain=\ymi:i2, dashed] {f2y(\x)};
		\addplot[black!30!red, mark=o, ultra thick, mark size = 4pt] coordinates{(i2,f2)};
	\end{axis}
\end{tikzpicture}
		}
	\caption{
		Illustrations of $l_1(t+1)$, $l_{-1}(t+1)$ and $-\frac{P}{P-1}\hat x(t+1)$ when $l_1(t) = 0$, $l_{-1}(t) = -\frac{P}{P-1}\hat x(t)^2$. The solid lines highlight the values of $y(t)$ where $l_i(t+1) \geq -\frac{P}{P-1}\hat x(t+1)^2$. We see that in (a) the solid lines do not overlap, i.e. given that the assumptions of Lemma~\ref{lemma:rec_cond} are fulfilled for some $t$, they will be fulfilled the next time step as well. In (b) the solid lines overlap, i.e. there are values for $y(t)$ so that the assumptions are violated the next time step.}
        \label{fig:quadfuns}
\end{figure}

\begin{theorem}[Certainty equivalence, upper bound]
	\label{thm:Certainty}
        Given a real number $a$ and a quantity $\gamma > 0$. Assume that
        \[
                P = \frac{1}{2} (1 - \gamma^2a^2) + \sqrt{\gamma^2(-1 + \gamma^2) + (\gamma^2a^2 -1)^2 / 4} > 1.
        \]
        If $P$ and $\gamma$ fulfill the \emph{curvature condition}~\eqref{eq:curvature_condition} and \emph{strong negativity condition}~\eqref{eq:strong_negativity} below, then the closed-loop system \eqref{eq:sys} controlled with the certainty-equivalence deadbeat controller \eqref{eq:certainty_equivalence} has gain from $(w, v) \to x$ bounded above by $\gamma$.
        \begin{equation}
                P > 2\gamma - 1
                \label{eq:curvature_condition}
        \end{equation}
        \begin{equation}
                (P + 2 \gamma^2 - 1) \left( P - 1 - 2\sqrt{\gamma^2 - P})^2\right) \geq (P-1)\left((P + 1)^2 - 4\gamma^2 \right)
                \label{eq:strong_negativity}
        \end{equation}
\end{theorem}
\begin{remark}
	We can solve \eqref{eq:strong_negativity} with equality restricted to the domain $P > 2\gamma -1$. The resulting $\gamma$ satisfies $(|a| + \sqrt{a^2 + 1})\sqrt{a^2 + 1} \leq \gamma \leq 2.1a^2 + 2$, and is shown in Figure~\ref{fig:gammascaling}.
\end{remark}
	\begin{remark}
		In \cite{Vinnicombe2004}, Vinnicombe studied the state-feedback version of the problem and found that the bound $\gamma = |a| + \sqrt{a^2 + 1}$ is achieved by the control policy
		\[
			u(t) = \begin{cases}
				ax(t),& \text{if } \alpha_1(t) \leq \alpha_{-1}(t) \\
				-ax(t),&\text{else},
			\end{cases}
		\]
		where $\alpha_b(t) = \sum_{\tau \leq t-1}\left( x(\tau + 1) - ax(\tau) - bu(\tau) \right)^2$. If we apply this control policy to the noisy measurements $y(t) = x(t) + v(t)$ we have that $x(t+1) = ax(t) + bu(t) + w(t) \pm av(t)$, and we get $\|x\|_2 \leq \gamma \|\bmat{1 & a}(w,v)\|_2 \leq (|a| + \sqrt{1+a^2})\sqrt{1 + a^2}\|(w,v)\|_2$ which is the lower bound in Figure~\ref{fig:gammascaling}. 
	\end{remark}
\begin{proof}[Theorem~\ref{thm:Certainty}, full]
	By assumption $P > 1$ is positive so Theorem~\ref{thm:full_state} applies. We will show that if the \emph{curvature condition} and the \emph{strong negativity condition} are fulfilled, then the assumptions in Lemma~\ref{lemma:rec_cond} will hold for all $t$. Then, by Theorem~\ref{thm:observer-based-feedback} the observer-based controller is finite-gain for the original system. For $t = 0$, we have that $l_i(0) = 0$, $\hat x(0) = 0$ and that $l_{i}(t) \leq -\frac{P}{P-1}\hat x(0)^2$ holds trivially. Fix $t \geq 0$, assume without loss of generality that $0 \geq l_1(t) \geq l_{-1}(t)$ and that $l_{-1}(t) \leq - \frac{P}{P-1}\hat x(t)^2$. By Lemma~\ref{lemma:rec_cond} $\max_i \{l_i(t+1) \} \leq 0$. It remains to show that
        \begin{equation}
                \min_i\{l_i(t+1)\} \leq -\frac{P}{P-1}\hat x(t+1)^2.
                \label{eq:min_cond}
        \end{equation}

        Let $z(t) := y(t) - \frac{P}{2\gamma^2}\hat x(t)$. Then $\hat x(t+1) = 2\hat g z(t)$ and using Proposition~\ref{prop:prop}, letting $X = P + \gamma^2 - 1$ we have
        \begin{align*}
                l_1(t+1) & = l_1(t) + \left(-\frac{P\hat x(t)}{2} + \gamma^2 z(t) \right)^2(1 / X - 1/\gamma^2) \\
                l_{-1}(t+1) &  = l_{-1}(t) +    \left(\frac{P\hat x(t)}{2} + \gamma^2 z(t) \right)^2 / X - \left(-\frac{P\hat x(t)}{2} + \gamma^2 z(t) \right)^2 / \gamma^2 - P\hat x(t)^2
        \end{align*}
        \paragraph{Curvature:} For \eqref{eq:min_cond} to be true for all $z(t)\in \R$ it is necessary that $l_i(t+1) + 4\frac{P}{P-1}\hat g^2 z(t)^2 $ is concave in $z(t)$. This is the case if and only if
        \begin{align}
                \gamma^4(1/X - 1/\gamma^2)  & \leq -4\frac{P}{P-1}\hat g^2 \label{eq:concavity}\\
                \iff \quad \gamma^4 & \geq -4\frac{P}{P-1}\frac{1}{1 / X - 1 / \gamma^2}\hat g^2 \nonumber
        \end{align}
        Insert $\hat g = \gamma^2 a^2 / X$  to get
        \[
                -4\frac{P}{P-1}\frac{1}{1 / X - 1 / \gamma^2}\hat g^2 = 4\frac{P}{P-1}\frac{\gamma^2 X}{X - \gamma^2}\hat g^2 = \frac{4P}{(P-1)^2} \gamma^2 a^2 / X \gamma^4.
        \]
        Further, insert
        \[
                P = \frac{1}{a^2 / X + \gamma^{-2}} \iff \frac{a^2}{X} = \frac{1}{P} - \gamma^{-2}
        \]
        to get
        \begin{equation}
                -4\frac{P}{P-1}\frac{1}{1 / X - 1 / \gamma^2}\hat g^2 = 4 \frac{\gamma^2 - P}{(P-1)^2}\gamma^4.
                \label{eq:simplified_rhs}
        \end{equation}
        The concavity condition \eqref{eq:concavity} simplifies to the \emph{curvature condition} \eqref{eq:curvature_condition},
        \begin{align*}
                1 \geq 4 \frac{\gamma^2 - P}{(P-1)^2} \iff (P+1)^2 \geq 4 \gamma^2 \iff \quad P  \geq 2\gamma - 1.
        \end{align*}
        \paragraph{Strong negativity:}
        Define the upper bounds
        \begin{align*}
                \bar l_1(t+1) & := \left(-\frac{P\hat x(t)}{2} + \gamma^2 z(t) \right)^2(1 / X - 1/\gamma^2) \\
                \bar l_{-1}(t+1) &  := -\frac{P}{P-1}\hat x(t)^2 +  \left(\frac{P\hat x(t)}{2} + \gamma^2 z(t) \right)^2 / X - \left(-\frac{P\hat x(t)}{2} + \gamma^2 z(t) \right)^2 / \gamma^2 - P\hat x(t)^2.
        \end{align*}
        Also define the sets
        \[
                \mathcal I_i := \left\{z \in \R : l_i(t+1) \geq -4\frac{P}{P-1}\hat g^2 z(t)^2\right \}.
        \]
	and $\bar{\mathcal I}_i$ anagolously. Then the inequality~\eqref{eq:min_cond} is satisfied if and only if $\#\left(\mathcal I_1 \cap \mathcal I_{-1}\right) \leq 1$. Since $\bar l_i \geq l_i$ we have that $\mathcal I_i \subseteq \bar{\mathcal I}_i$, and a sufficient condition is that they intersection contains at most one point, i.e. $\#\left(\bar{\mathcal I}_1 \cap \bar{\mathcal I}_{-1}\right) \leq 1$. The reason we allow for the intersection to contain one point, is that at such a point both $l_1(t+1)$ and $l_{-1}(t+1)$ fulfills \eqref{eq:min_cond} with equality.  We will start with characterizing $\bar{\mathcal I}_1$ by looking for the solutions to $\bar l_1(t+1) = -4 \frac{P}{P-1}\hat g^2 z(t)^2$:
        \[
        \begin{split}
		& \left(-\frac{P\hat x(t)}{2} + \gamma^2 z(t) \right)^2(1 / X - 1/\gamma^2)  = -4 \frac{P}{P-1}\hat g^2 z(t)^2 \\
		\iff &\left(- \frac{P\hat x(t)}{2} + \gamma^2z(t)\right)^2  = 4 \frac{\gamma^2 - P}{(P-1)^2}(\gamma^2 z(t))^2 \\
		\iff &\left(- \frac{P\hat x(t)}{2} + \gamma^2\left(1 + 2 \frac{\sqrt{\gamma^2 - P}}{P-1}\right) z(t)\right)\left(- \frac{P\hat x(t)}{2} + \gamma^2\left(1 - 2 \frac{\sqrt{\gamma^2 - P}}{P-1}\right) z(t)\right)  = 0
        \end{split}
\]
        We conclude that for positive $\hat x(t)$
        \[
                \bar{\mathcal I}_1 = \left[\frac{P}{2\gamma^2} \left(1 + 2 \frac{\sqrt{\gamma^2 - P}}{P-1}\gamma^2 z(t)^2\right)^{-1}\hat x(t), \frac{P}{2\gamma^2} \left(1 - 2 \frac{\sqrt{\gamma^2 - P}}{P-1}\gamma^2 z(t)^2\right)^{-1}\hat x(t) \right].
        \]

        We continue with the solutions to $\bar l_2(t+1) = - 4\frac{P}{P-1}\hat g^2 z(t)^2$.
       \begin{multline*}
-\frac{P}{P-1}\hat x(t)^2 +  \left(\frac{P\hat x(t)}{2} + \gamma^2 z(t) \right)^2 / X - \left(-\frac{P\hat x(t)}{2} + \gamma^2 z(t) \right)^2 / \gamma^2 - P\hat x(t)^2  \\=-4 \frac{P}{P-1}\hat g^2 z(t)^2
	\end{multline*}	
        Using \eqref{eq:simplified_rhs} we get

	\begin{align*}
		\iff & \left(\frac{1}{X} - \frac{1}{\gamma^2} \right)\left( 1 - 4\frac{\gamma^2 - P}{(P-1)^2} \right) \left(\gamma^2 z(t) \right)^2 + \left( \frac{1}{X} + \frac{1}{\gamma^2}\right)P\hat x(t) \gamma^2 z(t) \\
		& + \left(\frac{1}{4} \left( \frac{1}{X} - \frac{1}{\gamma^2} \right) - \frac{1}{P-1} \right) \left (P\hat x(t) \right)^2 = 0 \\
		\iff & \left( z(t) \right)^2 - \frac{X + \gamma^2}{X - \gamma^2}\frac{(P-1)^2}{(P-1)^2 - 4(\gamma^2 - P)}P\hat x(t) \gamma^2 z(t) \\
		& + \frac{\frac{1}{4} - \frac{1}{P-1}\frac{1}{1 / X - 1 / \gamma^2}}{(P-1)^2 - 4(\gamma^2 - P)}(P-1)^2 P^2 \hat x(t)^2 = 0\\
		\iff & \left( \gamma^2 z(t) \right)^2 - \frac{(P + 2\gamma^2 - 1)(P-1)}{(P+1)^2 - 4 \gamma^2}P\hat x(t) \gamma^2 z(t) \\
		& + \frac{1}{4}\frac{(P-1)^2 + 4\gamma^2(P + \gamma^2 - 1)}{(P+1)^2 - 4\gamma^2} P^2 \hat x(t)^2 = 0 \\
		\iff & \left(\gamma^2 z(t) - \frac{1}{2}\frac{(P + 2\gamma^2 - 1)(P-1)}{(P+1)^2 - 4\gamma^2}P\hat x(t) \right)^2 \\
		& - (P + 2\gamma^2 - 1)^2\frac{\gamma^2 - P}{\left((P+1)^2 - 4\gamma^2\right)^2} P^2 \hat x(t)^2 = 0
	\end{align*}

        which has the solutions
        \[
                z(t) = \frac{1}{2\gamma^2}(P + 2\gamma^2 - 1)\frac{P - 1 \pm 2\sqrt{\gamma^2 - P}}{(P+1)^2 - 4\gamma^2}P\hat x(t).
        \]
        Thus for positive $\hat x(t)$,
        \begin{multline*}
                \bar{\mathcal I}_{-1} = \Bigg[ \frac{1}{2\gamma^2}(P + 2\gamma^2 - 1)\frac{P - 1 - 2\sqrt{\gamma^2 - P}}{(P+1)^2 - 4\gamma^2}P\hat x(t), \\
                \frac{1}{2\gamma^2}(P + 2\gamma^2 - 1)\frac{P - 1 + 2\sqrt{\gamma^2 - P}}{(P+1)^2 - 4\gamma^2}P\hat x(t)\Bigg]
        \end{multline*}
	From the definition, it is clear that the vertex of $\bar l_1(t+1)$ lies closer to the origin, than that of $\bar l_{-1}(t+1)$. Thus $\#\left(\bar{\mathcal I}_1 \cap \bar{\mathcal I}_2\right) \leq 1$ is equivalent to
        \[
                \frac{P}{2\gamma^2} \left(1 - 2 \frac{\sqrt{\gamma^2 - P}}{P-1}\gamma^2 z(t)^2\right)^{-1}\hat x(t) \leq \frac{1}{2\gamma^2}(P + 2\gamma^2 - 1)\frac{P - 1 - 2\sqrt{\gamma^2 - P}}{(P+1)^2 - 4\gamma^2}P\hat x(t),
        \]
	which simplifies to \eqref{eq:strong_negativity}. The case when $\hat x(t)$ is negative is similar.
\end{proof}

\begin{figure}
        \centering
        \begin{tikzpicture}
	\begin{scope}
	\begin{axis}[%
			xlabel = {$a$},
			ylabel = {$\gamma$},
			xmin = -6,
			xmax = 6,
			width = 6cm,
			line width=0.5mm,
			scale=1.0,
			]

			\addplot[line width=0.7mm, black] table[x=x, y=gamma, col sep=comma] {figures/gammas.csv};
			\addplot[domain=-10:10, variable=\x, black!30!blue, dotted] plot ({\x}, {2.1*\x*\x + 2});
			\addplot[domain=-10:10, variable=\x, black!30!red, dashed] plot ({\x}, {(abs(\x) + sqrt(\x*\x + 1))*sqrt(\x*\x + 1)});
			\draw[black] (axis cs:-1.5, 0) rectangle (axis cs:1.5, 7);
	\end{axis}
	\end{scope}

	\begin{scope}[xshift=6cm]
	\begin{axis}[%
			xlabel = {$a$},
			ylabel = {$\gamma$},
			xmin = -1,
			xmax = 1,
			ymin = .3,
			width = 6cm,
			line width=0.5mm,
			]

			\addplot[line width=0.7mm, black] table[x=x, y=gamma, col sep=comma] {figures/gammas.csv};
			\addplot[domain=-2:2, variable=\x, black!30!blue, dotted] plot ({\x}, {2.1*\x*\x + 2});
			\addplot[domain=-2:2, variable=\x, black!30!red, dashed] plot ({\x}, {(abs(\x) + sqrt(\x*\x + 1))*sqrt(\x*\x + 1)});

			\node (gam) at (axis cs:-0.5, 4)  {$\gamma$};
			\node[color = black!30!blue] (gamhi) at (axis cs:0.5, 4) {$2.1 a^2 + 2$};
			\node[color = black!30!red] (gamlo) at (axis cs:-0, 0.7) {$(|a| + \sqrt{a^2 + 1})\sqrt{a^2 + 1}$};

			\draw[black, -latex, line width=0.3mm] (axis cs:-.7,2.6) -- (gam);
			\draw[black!30!blue, dotted, -latex, line width=0.3mm] (axis cs:.7,3.1) --(gamhi);
			\draw[black!30!red, dashed, -latex, line width=0.3mm, bend right] (axis cs:.6,1.9) to [out=30, in=200] ([xshift = 7mm, yshift=-2mm]gamlo.north);
	\end{axis}
	\end{scope}

	\draw[black, thick, -] (2.2,.25) -- (6,0);
	\draw[black, thick, -] (2.2,.53) -- (6,3.61);
\end{tikzpicture}
	\caption{Guaranteed bound on the $\ell_2$-gain from disturbances to error under feedback with the certainty equivalence controller with respect to $a$. We note that experimentally $\gamma$ is lower bounded by $(|a| + \sqrt{a^2 + 1})\sqrt{a^2 + 1}$ and upper bounded by $\leq 2.1a^2 + 2$. The lower bound becomes tighter as $a$ increases.}
        \label{fig:gammascaling}
\end{figure}
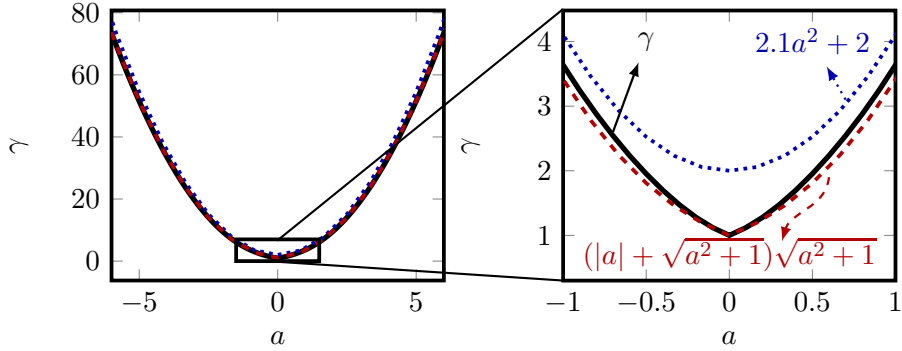

\section{Conclusions}
	This article presents a constructive approach to accounting for worst-case models of measurement noise, disturbance and uncertain parameters in controller design. In particular Theorem~\ref{thm:observer-based-feedback} shows that it is necessary and sufficient to consider feedback from the current states of a finite set of observers and cumulative performance measures. The performance measures compress the history allowing the controller to learn from past data. In Section~\ref{sec:Certainty}, we used this constructive approach to extend the results of \cite{Vinnicombe2004} to the case of noisy measurements. We focused on scalar systems, but Theorems \ref{thm:full_state} and \ref{thm:observer-based-feedback} can easily be extended to MIMO systems. In particular, we are excited about the potential in extending Minimax Adaptive Control~\cite{Rantzer2021} to the output feedback case.
\acks{This project has received funding from the European Research Council (ERC) under the European Union's Horizon 2020 research and innovation programme under grant agreement No 834142 (ScalableControl). The authors are thankful to their colleagues Bo Bernhardsson and Venkatraman Renganathan (Department of Automatic Control, Lund University) for help in reviewing and revising earlier versions of the manuscript.}

\appendix
\bibliography{references}

\end{document}